\documentclass{amsart}

\usepackage[T1]{fontenc}
\usepackage[utf8x]{inputenc}
\usepackage{lmodern, euscript}
\usepackage{enumerate, graphics}
\usepackage{mathabx}
 \usepackage[all,cmtip]{xy}

\usepackage{amsfonts,amsmath,amstext,amsbsy,amssymb,amsbsy,
  amsopn,amsthm,amscd} 
\usepackage[T1]{fontenc}
\usepackage{hyperref}
\usepackage{mathrsfs}

\RequirePackage[dvipsnames]{xcolor} 
\definecolor{halfgray}{gray}{0.55} 
\definecolor{webgreen}{rgb}{0,0.5,0}
\definecolor{webbrown}{rgb}{.6,0,0} \hypersetup{%
  colorlinks=true, linktocpage=true, pdfstartpage=3,
  pdfstartview=FitV,%
  breaklinks=true, pdfpagemode=UseNone, pageanchor=true,
  pdfpagemode=UseOutlines,%
  plainpages=false, bookmarksnumbered, bookmarksopen=true,
  bookmarksopenlevel=1,%
  hypertexnames=true,
  pdfhighlight=/O,
  urlcolor=webbrown, linkcolor=RoyalBlue,
  citecolor=webgreen, 
  pdftitle={Recap},%
  pdfsubject={},%
  pdfkeywords={},%
  pdfcreator={pdfLaTeX},%
  pdfproducer={LaTeX with hyperref}%
}

\newtheorem{theorem}{Theorem}[section]

\newtheorem{lemma}[theorem]{Lemma}

\newtheorem{proposition}[theorem]{Proposition}
\theoremstyle{definition}

\newtheorem{remark}[theorem]{Remark}

\newtheorem{claim}[theorem]{Claim}

\newcommand{\field}[1]{\mathbb{#1}}
\newcommand{\R}{\field{R}}

\newcommand{\Z}{\field{Z}}

\newcommand{\T}{\field{T}}

\newcommand{\cA}{\mathcal A}

\newcommand{\cF}{\mathcal{F}}

\newcommand{\cU}{\mathcal{U}}
\newcommand{\cV}{\mathcal{V}}

\newcommand{\cT}{\EuScript T}
\newcommand{\euP}{\EuScript P}

\newcommand{\eg}{{\it e.g., } }

\renewcommand{\phi}{\varphi}

\newcommand{\eps}{\varepsilon}

\renewcommand{\|}{\,\Vert\,}

\begin{document}
\baselineskip=14pt

\title{Smooth rigidity for codimension one Anosov flows}
\author {Andrey Gogolev and Federico Rodriguez Hertz}\thanks{The authors were partially supported by NSF grants DMS-1955564 and DMS-1900778, respectively}

 \address{Department of Mathematics, The Ohio State University,  Columbus, OH 43210, USA}
\email{gogolyev.1@osu.edu}

\address{Department of Mathematics, The Pennsylvania State University,  University Park, PA 16802, USA}
\email{hertz@math.psu.edu}

\begin{abstract} 
  \begin{sloppypar}
  We introduce the matching functions technique in the setting of Anosov flows. Then we observe that simple periodic cycle functionals (also known as temporal distances) provide a source of matching functions for conjugate Anosov flows. For conservative codimension one Anosov flows $\phi^t\colon M\to M$, $\dim M\ge 4$, these simple periodic cycle functionals are $C^1$ regular and, hence, can be used to improve regularity of the conjugacy. Specifically, we prove that a continuous conjugacy must, in fact, be a $C^1$ diffeomorphism for an open and dense set of codimension one conservative Anosov flows.

  \end{sloppypar}
\end{abstract}
\maketitle

\section{Introduction}

Let $M$ be a closed smooth Riemannian manifold. Recall that a smooth flow $\phi^t\colon M\to M$ is called {\it Anosov} if the tangent bundle admits a $D\phi^t$-invariant splitting $TM=E^s\oplus X\oplus E^u$, where $X$ is the generator of $\phi^t$, $E^s$ is uniformly contracting and $E^u$ is uniformly expanding under $\phi^t$. Basic examples of Anosov flows are geodesic flows in negative curvature and suspension flows of Anosov diffeomorphisms. If $\dim E^s=1$ then $\phi^t$ is called a {\it codimension one} Anosov flow.

Anosov flows $\phi_1^t\colon M\to M$ and $\phi_2^t\colon M\to M$ are called {\it orbit equivalent} if there exists a homeomorphism $h\colon M\to M$ which sends orbits of $\phi_1^t$ to the orbits of $\phi_2^t$ preserving the time direction. A much stronger equivalence property for flows is conjugacy. Flows $\phi_1^t$ and $\phi_2^t$ are {\it conjugate} via a homeomorphism $h$ if $h\circ \phi_1^t=\phi_2^t\circ h$ for all $t\in \R$. 

Anosov's structural stability asserts that if $\phi_1^t$ is a transitive Anosov flow and $\phi_2^t$ is a sufficiently small $C^1$ perturbation of $\phi_1^t$ then $\phi_2^t$ is orbit equivalent to $\phi_1^t$. An orbit equivalence $h$ can be improved to a conjugacy by adjusting along the flow lines if and only if the periods of corresponding periodic orbits coincide, that is, $per_{\phi_1^t}(\gamma)=per_{\phi_2^t}(h(\gamma))$ for all periodic orbits $\gamma$~\cite[Theorem 19.2.9]{KH}. Indeed, one can begin with orbit equivalence, then, if periods match, the derivative of the orbit equivalence along the flow is cohomologous to constant 1 by the Livshits theorem. Then one use the transfer function to adjust the orbit equivalence into another orbit equivalence whose derivative along the flow is 1, which is then a conjugacy. According to~\cite{FO} this application of the Livshits theorem is due to A. Katok.

It is also well-known that a conjugacy between Anosov flows must be bi-H\"older continuous. Generally speaking, a better regularity cannot be expected as can be seen from the example of constant roof suspensions of two Anosov diffeomorphisms, which are not $C^1$ conjugate, but merely H\"older conjugate. However, it is the authors' belief that aside from this special case a lot more rigidity can be expected. In the setting of 3-dimensional contact Anosov flows Feldman and Ornstein proved that any continuous conjugacy is, in fact, smooth~\cite{FO}. We proceed to present more evidence to support our belief.

\begin{theorem}
\label{thm_flows}
Let $\phi_1^t\colon M\to M$ and $\phi_2^t\colon M\to M$ be volume preserving Anosov flows, which are conjugate via a homeomorphism $h\colon M\to M$. Assume that $\dim M=4$ and that there exists a periodic point $p=\phi_1^T(p)$ such that the linearized return map $D\phi_1^T\colon T_pM\to T_pM$ has a pair of (non-real) complex conjugate eigenvalues. Then at least one of the following holds
\begin{enumerate}
\item $\phi_1^t$ and $\phi_2^t$ are constant roof suspensions over Anosov diffeomorphisms;
\item conjugacy $h$ is $C^\infty$ smooth;
\end{enumerate}
\end{theorem}
The above theorem applies to abstract 4-dimensional volume preserving Anosov flows. Recall that Verjovsky conjecture states that any codimension one Anosov flow on a manifold of dimension greater than three is orbit equivalent to a suspension flow of a toral automorphism. If Verjovsky conjecture is true then $\phi_i^t$ must, in fact, be suspensions of Anosov diffeomorphisms of $\T^3$. We also prove the following result in any dimension $\ge4$.
\begin{theorem}
\label{thm_flows2}
Let $M$ be a closed manifold of dimension at least 4. Denote by $\cU$ the space of codimension one  volume preserving Anosov flows on $M$. Then there exists a $C^1$ open and $C^\infty$ dense subset $\cV\subset\cU$ such that if $\phi_1^t\in\cV$ and $\phi_2^t\in \cU$ are conjugate via a homeomorphism $h\colon M\to M$ then $h$ is, in fact, a $C^1$ diffeomorphism.
\end{theorem}

\begin{remark}
In both Theorem~\ref{thm_flows} and Theorem~\ref{thm_flows2}, the assumption on the conjugacy can be weakened by only assuming that time-1 maps of the flows are conjugate. To see that recall that any Anosov flow which is not a constant roof suspension is topologically weakly mixing~\cite{Pl}. Hence, we can assume that $\phi^t_1$ is topologically weakly mixing. Then it is not hard to check that orbits whose period is irrational are dense in $M$ (see, \eg~\cite[Lemma~2.2]{BG}). The restriction of the conjugacy for time-1 maps on these irrational periodic orbits must be conjugacy for the flows, because the time-1 map orbits are dense in the irrational periodic orbits. From here, we see the time-1 map conjugacy is a conjugacy of flows on a dense set and hence is, in fact, conjugacy of flows.
\end{remark}

\begin{remark}
Analogous results, with the same proofs, also hold in the setting of partially hyperbolic diffeomorphisms. Namely, one has to consider volume preserving partially hyperbolic diffeomorphisms with one-dimensional stable subbundle, one-dimensional isometric center subbundle and a higher dimensional ($\ge 2$) unstable subbundle. Then one can conclude regularity of the conjugacy for an open and dense set of such diffeomorphisms. It is an interesting problem to generalize to the setting with higher dimensional isometric center subbundle. We would like to thank Danijela Damjanovi\'c for this remark.
\end{remark}

The proofs of both theorems rely on matching functions along the higher dimensional invariant foliation. We have previously developed such matching functions technique for expanding maps~\cite{GRH} and for codimension one Anosov diffeomorphisms~\cite{GRH3}. By consistently working in $C^{1+\eps}$ category one can, in fact, conclude that the conjugacy in Theorem~\ref{thm_flows2} is a $C^{1+\eps}$ diffeomorphism for some small $\eps>0$. In the setting of Theorem~\ref{thm_flows}, the presence of the ``conformal'' periodic unstable leaf allows to employ rather standard bootsrtap arguments to upgrade the regularity of the conjugacy from $C^1$ to $C^\infty$. The authors do not know how to further bootstrap regularity of the conjugacy in the setting of Theorem~\ref{thm_flows2}.

Finally, we would like to point out that, in the case when the flows are suspensions, the methods of~\cite{GRH3} become directly applicable. Specifically, recall the following result from~\cite{GRH3}.

\begin{theorem}[Theorem 1.7, \cite{GRH3}]
\label{thm_codim1}
Let $L\colon\T^d\to \T^d$ be a generic hyperbolic automorphism with one dimensional stable subspace. Assume that
$$
(\log\mu)^2-(\log\xi_l)^2>\log\mu(\log\xi_l-\log\xi_1)
$$
where $\mu^{-1}$ is the absolute value of the stable eigenvalue, $\xi_1$ is the smallest absolute value of the eigenvalues which are greater than 1 and $\xi_l$ is the largest absolute value of the eigenvalues of $L$.

Then there exists a $C^1$ neighborhood $\cU$ of $L$ in $\textup{Diff}^r(\T^d)$, $r\ge 3$, and a $C^r$-dense $C^1$-open subset $\cV\subset\cU$ such that if $f_1,f_2\in\cV$ have matching Jacobian periodic data then $f_1$ and $f_2$ are $C^{1+\eps}$ conjugate for some $\eps>0$.
\end{theorem}

Following the proof of the above theorem and using the roof function instead of using Jacobians an interested reader will also be able to establish the following result.

\begin{proposition} Let $L\colon\T^d\to\T^d$ be an Anosov diffeomphism satisfying the assumptions of Theorem~\ref{thm_codim1} and let $\phi^t$ be a suspension flow of $L$. Then there exists a $C^1$ neighborhood $\cU$ of $\phi^t$ in the space of volume preserving flows and a $C^1$-open subset $\cV\subset\cU$ such that if $\phi_1^t\in\cV$ and $\phi_2^t\in\cU$ are conjugate then they are, in fact, $C^{1+\eps}$ conjugate for some $\eps>0$.
\end{proposition}

\subsection{Organization} 
In the next section we recall some facts about regularity of invariant foliations. Then we introduce the matching function machinery for Anosov flows. In Section~3 we give a rather self-contained proof of Theorem~\ref{thm_flows}, except for the bootstrap from $C^1$ to $C^\infty$ which was done in~\cite{GRH3}. In Section~4 we prove Theorem~\ref{thm_flows2}.

\section{Preliminaries}

\subsection{Regularity of invariant foliations}

Let $\phi^t$ be a codimension Anosov flow. The regularity of invariant subbundles is a well-studied in the literature~\cite{HPS, Hass}. In particular, the weak stable distribution $E^{0s}$ is $C^\nu$ if there exists a sufficiently large $t$ such that
$$
\sup \|D\phi^t(v^s)\|\cdot\|D\phi^{-t}(v_1^u)\|\cdot\|D\phi^t(v_2^u)\|^\nu<1
$$
where the supremum is taken over all unit vectors $v^s\in E^s(x),  v_2^u\in E^u(x), v_1^u\in E^u(\phi^t(x))$, $x\in M$.
The stable distribution is $C^\nu$ if there exists a sufficiently large $t$ such that
$$
\sup \|D\phi^t(v^s)\|\cdot\|D\phi^t(v^u)\|^\nu<1
$$
where the supremum is taken over all unit vectors $v^s\in E^s(x), v^u\in E^u(x)$, $x\in M$. Note that the second condition is stronger then the first one. We can verify this condition if $\phi^t$ is a codimension one volume preserving Anosov flow on a manifold of dimension $\ge 4$. 

Using volume invariance we have
$$
 \|D\phi^t(v^s)\|\cdot\|D\phi^t(v^u)\|<\|D\phi^t(v^s)\|J^{u}\phi^t=J^s\phi^t J^{u}\phi^t=1,
$$
where $J^s\phi^t $ and $J^{u}\phi^t$ denote the stable and the unstable Jacobians of $\phi^t$.
By compactness the supremum the above expression is also $<1$. Therefore $E^s$ and $E^{0s}$ are $C^1$. Hence, the stable foliation $W^s$ and and the  weak stable foliation $W^{0s}$ are also uniformly $C^1$. (In fact, because the inequality is strict one can also conclude $C^\nu$ regularity for some $\nu>1$.) Reversing the time and using the same criterion one can also see that the weak unstable foliation $W^{0u}$ is uniformly $C^1$, but the unstable $W^u$ is merely H\"older continuous transversely to the leaves.

\subsection{Matching functions and the Subbundle Theorem}

Now let $\phi_i^t\colon M\to M$, $i=1,2$, be codimension one, volume preserving Anosov flows which are conjugate, $h\circ \phi_1^t=\phi_2^t\circ h$. We will consistently use the subscript $i\in\{1,2\}$ for distribution and foliations of the flow $\phi_i^t$ to indicate dependence on the flow.

Here we will explain a certain construction of sub-bundles $E_i$ of the unstable bundles $E_i^u$ via locally matching functions.

For each $x$ consider pairs of $C^1$ functions $(\rho^1,\rho^2)$ where $\rho^1$ is defined on an open neighborhood of $x$ in $W^{0u}_1(x)$, $\rho^2$ is defined on an open neighborhood of $h(x)$ in $W^{0u}_2(h(x))$ and such that
$$
\rho^1=\rho^2\circ h
$$
This relation is what we call a {\it matching relation.} We collect all such pairs into a space $V_x$
$$
V_x=\{(\rho^1,\rho^2): \rho^1=\rho^2\circ h\}
$$
The domains of definition of $\rho^1$ and $\rho^2$ can be arbitrarily small open sets. Also denote by $V_{x,1}$ the collection of all possible $\rho^1$, that is, projection of $V_x$ on the first coordinate, and by $V_{x,2}$ the projection on the second coordinate.

Now we can define linear subspaces $E_i(x)\subset E^{0u}_i(x)$ by intersecting kernels of all $D\rho^i$ at $x$, $i=1,2$. Namely,
$$
E_i(x)=\bigcap_{ \rho^i\in V_{x,i}} \ker D\rho^i(x)
$$
Also let $m_i(x)=\dim E_i(x)$ and let $m_i=\min_{x\in M} m_i(x)$. 

Conjugacy relation implies that if $(\rho^1,\rho^2)\in V_{\phi_1^t(x)}$ then $(\rho^1\circ \phi_1^t,\rho^2\circ \phi_2^t)\in V_{x}$. It immediately follows that $E_i$ is invariant under $D\phi_i^t$. Also, we can check that $E_i(x)$ is invariant under the stable holonomy. Indeed, let $b\in W^{s}_1(a)$ and let 
$Hol_{a,b}\colon W^{0u}_{1,loc}(a)\to W^{0u}_{1,loc}(b)$ be the stable holonomy given by sliding points along the stable foliation $W^{s}_1$ with $Hol_{a,b}(a)=b$. Similarly, $Hol_{h(a),h(b)}\colon W^{0u}_{2,loc}(h(a))\to W^{0u}_{2,loc}(h(b))$ is the stable holonomy for $\phi_2^t$. Let $(\rho^1,\rho^2)\in V_b$, then, because the stable holonomy is $C^1$, and the conjugacy sends the stable foliation of $\phi_1^t$ to the stable foliation of $\phi_2^t$ we can conclude that $(\rho^1\circ Hol_{a,b}, \rho^2\circ Hol_{h(a),h(b)})\in V_a$. From this and the definition of the spaces $E_i$ we can see that $E_1(b)=D\, Hol_{a,b} E_1(a)$ and $E_2(h(b))=D\, Hol_{h(a),h(b)} E_2(h(a))$. In particular, we conclude that the level sets of the dimension function $\{x: m_i(x)=k \}$, are $W^{s}_i$-saturated. But they are also saturated by the flow lines, hence, $W^{0s}_i$-saturated.

It is not hard to verify that $m_i\colon M\to\Z_+$ is an upper semi-continuous function. To see that one can write $E_i(x)$ as a finite intersection (due to $E_i(x)$ being a finite dimensional subspace)
$$
E_i(x)=\bigcap_{j=1}^K \ker D\rho^i_j(x)
$$
with minimal $K$. Then, we have that $\ker D\rho^i_j(x)$ depend continuously on $x$ due to $C^1$ regularity of $\rho_i^j$. Then for all $y$ which are sufficiently close to $x$
$$
E_i(y)\subset \bigcap_{j=1}^K \ker D\rho^i_j(y)
$$
and we have
$$
m_i(y)=\dim E_i(y)\le \dim \bigcap_{j=1}^K \ker D\rho^i_j(y)=\dim \bigcap_{j=1}^K \ker D\rho^i_j(x)=\dim E_i(x)=m_i(x)
$$
which is indeed the upper semi-continuity property for $\mathbb Z$-valued funciton.

Therefore we have that the set $\{x: m_i(x)=m_i \}=\{x: m_i(x)<m_i+\frac12 \}$ is open, non-empty and saturated by the minimal foliation $W^{0s}_i$, which means that $\{x: m_i(x)=m_i \}=M$ and we have well-defined $C^1$ distributions $E_i\subset E_i^u$. We have the following Subbundle Theorem, which is a direct analogue of the Subbundle Theorem for Anosov diffeomorphisms~\cite[Theorem 4.1]{GRH3}.

\begin{theorem}[Subbundle Theorem]
\label{thm_tech}
Let $\phi_i^t\colon M\to M$, $i=1,2$, be conjugate Anosov flows, $h\circ \phi_1^t=\phi_2^t\circ h$. Assume that both flows have $C^1$ stable foliations.
Then there exist $C^1$ regular, $D\phi_i^t$-invariant distributions $E_i\subset E_i^u$, such that
\begin{enumerate}
\item distributions $E_i$ integrate to $\phi_i^t$-invariant foliations $W_i\subset W_i^u$;
\item the distribution $E^s_i\oplus E_i$ integrates to an $\phi_i^t$-invariant $C^1$ foliation 
which is subfoliated by both $W^s_i$ and $W_i$;
\item conjugacy $h$ maps $W_1$ to $W_2$;
\item the restrictions of $h$ to the unstable leaves are uniformly $C^1$ transversely to
$W_1$;
\item if $(\rho^1,\rho^2)\in V_x$ is a matching pair then $\rho_i$ is constant on connected local leaves of $W_i$;
\end{enumerate}
\end{theorem}

First we need to verify that $E_i(x)\subset E_i^u(y)$. To see that consider $\rho_i(y)$ to be the local time it takes for $y$ to arrive on the local unstable manifold of $x$, that is,  $\rho_i(y)$, $y\in W^{0u}_i(x)$, is defined by
$$
\phi^{\rho_i(y)}_i\in W^u_{i,loc}(x)
$$
Then, clearly $\ker D\rho_i (x)=E_i^u(x)$ and, hence, $E_i(x)\subset E_i^u(y)$, $i=1,2$.

The proof of the first item is based on the fact that finite intersections of level sets of matching functions give the integral submanifolds of $E_i$. This argument is omitted as it is exactly the same as the argument in~\cite{GRH3}. The second item comes from the fact that $E_i$ and $\cF_i$ are invariant under the stable holonomy. Indeed, from the discussion preceding the statement of the Subbundle Theorem we have that $V_{b,i}\circ Hol_{a,b}=V_{a,i}$, $a\in W^s_i(b)$, hence we conclude that $DHol_{a,b} E_i(a)=E_i(b)$ and $Hol_{a,b}(\cF_{i, loc}(a))=\cF_{i,loc}(b)$, where the latter precisely means joint integrability of $\cF_i$ and $W^s_i$. We remark that $C^1$ regularity of the stable foliation is crucial for $E_i$ to be well-defined and to satisfy the second property.

The rest of these properties are verified in verbatim the same way as in the discrete time case~\cite{GRH3}, hence, we refrain from repeating these arguments here. 
\section{Proof of Theorem~\ref{thm_flows} on rigidity of 4-dimensional volume preserving Anosov flows}

We will give a self-contained proof for $C^1$ regularity. Using the Subbundle Theorem one can make a shorter proof, but we opt for a slightly longer exposition in order for the proof to be self-contained. At the end of the proof we will refer to~\cite{GRH3} for the bootstrap from $C^1$ to $C^\infty$.

Without loss of generality we can assume that the stable foliation is one-dimensional and the unstable foliation is two-dimensional. Then $D\phi_1^T(p)|_{E_{1}^u(p)}$ has a pair of (non-real) complex conjugate eigenvalues. We will denote by $W^{s}_i$, $W^{u}_i$, $W^{0s}_i$ and $W^{0u}_i$ the stable, unstable, weak stable and weak unstable foliations of $\phi_i^t$, respectively, $i=1,2$.

We begin with the definition of the simple PCFs (periodic cycle functionals) for Anosov flows $\phi_i^t\colon M\to M$.  Consider the holonomy along the stable foliation
$
Hol_{a,b}\colon W_{i,loc}^{0u}(a)\to W_{i,loc}^{0u}(b)
$
which takes $a$ to $b$. Then for any point $x\in W^u_{i,loc}(a)$ consider $Hol_{a,b}(x)$ and let $y$ be the unique point on $W^u_{i, loc}(b)$ such that $y$ and $Hol_{a,b}(x)$ belong to the same short orbit segment, in other words, $y=W^u_{loc}(b)\cap W^{0s}_{loc}(Hol_{a,b}(x))$. Define {\it simple PCF} $\rho^i_{a,b}\colon W^u_{i,loc}(a)\to\R$ as the flow time from $y$ to $Hol_{a,b}(x)$, that is,
$$
\phi_i^{\rho^i_{a,b}(x)}(y)=Hol_{a,b}(x)
$$
(The quantity $\rho^i_{a,b}(x)$ is also known as {\it temporal distance.})

\begin{figure}[h]
  \includegraphics{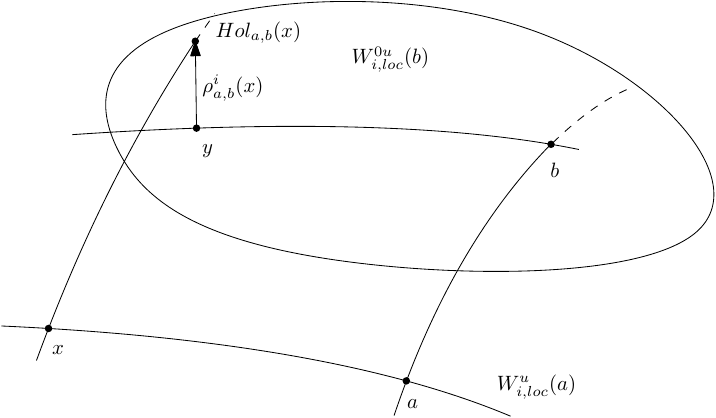}
  \caption{Simple PCF.}
\end{figure}

\begin{remark}
The above definition corresponds to PCFs with potential equal to 1 (hence the term ``simple''), so we omitted the dependence on the potential function. Also, unlike general PCFs, such simple PCFs can be defined in the geometric way by measuring the amount of non-intergrability between the stable and unstable foliations as explained above. (For a more general definition with arbitrary H\"older potential and basic properties of PCFs for Anosov flows see~\cite{GRH2}.) Let $a\in M$ and $b\in W^s_i(a)$.
We will need the following properties of simple PCFs.
\end{remark}

\begin{enumerate}
\item $\rho_{a,b}^i$ are conjugacy invariant, that is,
$$
\rho^1_{a,b}=\rho^2_{h(a),h(b)}\circ h
$$
Indeed this immediately follows from the definition and the fact that the conjugacy takes the stable and unstable foliations of $\phi_1^t$ to the stable and unstable foliations of $\phi_2^t$.
\item If $\rho_{a,b}^i\equiv0$ for all $a\in M$ and $b\in W^s_i(a)$ then foliations $W^s_i$ and $W^u_i$ jointly integrate to a codimension 1 foliation. This is also immediate from the definiton because $\rho_{a,b}^i\equiv0$ precisely means that the stable holonomy preserves the unstable leaves.
\item $\rho_{a,b}^i\colon W^u_{i,loc}(b)\to\R$ are uniformly $C^1$. Indeed, notice that $\rho_{a,b}^i$ is given by the time needed to flow from the image $Hol_{a,b}(W^u_{i,loc}(a))$ to $W^u_{i,loc}(b)$. The manifolds $W^u_{i,loc}(a)$ and $W^u_{i,loc}(b)$ are smooth, hence, $\rho_{a,b}^i$ is as smooth as the stable holonomy. 
\end{enumerate}

Now we can proceed with the proof by considering two cases.

\subsection{Case I: vanishing of PCFs} Assume that all $\rho^1_{a,b}\equiv0$ for all $a$ and $b$. Then by property~2 above $W^s_1$ and $W^u_1$ integrate together. Hence, one can apply a theorem of Plante~\cite[Theorem 3.1]{Pl} to conclude that $\phi_1^t$ is a suspension of a 3-dimensional Anosov diffeomorphism. By Newhouse's result on classification of codimension one Anosov diffeomorphisms, this diffeomorphism must live on $\T^3$ and must induce a hyperbolic automorphism on $H_1(\T^3;\R)$ by the Franks-Manning classification. Hyperbolicity of  induced map on  $H_1(\T^3;\R)$ easily implies that $H_1(M,\R)\simeq \R$, which allows to apply another result of Plante~\cite[Section~3]{Pl} which says that in this case the suspension is, in fact, a constant roof suspension. Hence in this case we obtain that the flows $\phi_i^t$ are constant roof suspensions. (Alternatively, instead of applying~\cite{Pl} in the last step one can deduce that the roof function is constant from vanishing of simple PCFs by using \cite[Proposition~2.2]{GRH3}.)

\subsection{Case II: non-vanishing PCFs} Now we assume that for some $a\in M$ and some $b\in W^s_1(a)$ we have $\rho_{a,b}^1\not\equiv 0$. From matching we also have $\rho_{h(a),h(b)}^2\not\equiv 0$. Then, because $\rho_{a,b}^1(a)=0$  we have that $\rho_{a,b}^1$ is non-constant and, hence, there exists $x_0\in W^u_{1,loc}(a)$ such that $D\rho_{a,b}^1(x_0)\neq 0$. Since $\rho_{a,b}^1$ is $C^1$ we also have $D\rho_{a,b}^1(y)\neq 0$ for all $y$ which are sufficiently close to $x_0$.


Now recall that we have a matching pair $(\rho^1_{a,b}, \rho^2_{h(a), h(b)})\in V_{x_0}$ with $D\rho^1_{a,b}(x_0)\neq 0$. Hence $m_1\le m_1(x_0)\le 1$. Similarly, we have $m_2\le 1$. (Recall the definition of $m_i$ from Section~2.2.)

First assume that $m_1=1$. If $m_1(x)=1$, then there exist $\rho^1\in V_{x,1}$ such that $D\rho_1(x)\neq 0$ and because $\rho^1$ is $C^1$ we also have $D\rho_1(y)\neq 0$ for all $y$ which are sufficiently close to $x$. This proves that the set $\{x: m_1(x)=1\}$ is a non-empty set which is open inside unstable leaves. But we also have that $\{x: m_1(x)=1\}$ is $W^{0s}_1$-saturated. Hence $\{x: m_1(x)=1\}$ is non-empty, open and $W^{0s}_1$-saturated. By minimality of $W^{0s}_1$ we conclude that $\{x: m_1(x)=1\}=M$. In particular $m_1(p)=1$, which implies that $E_1(p)\subset E_1^u(p)$ is a one-dimensional subspace invariant under $D\phi_1^T|_{E_1^u(p)}$, where $T$ is the period of $p$. But by our assumption $D\phi_1^T|_{E_1^u(p)}$ doesn't have real eigenvalues. Hence we obtain a contradiction in this case and conclude that $m_1=0$.

If $m_1(x)=0$ then, by applying the same argument as in the previous paragraph, to a pair of functions $\rho^1, {\bar \rho^1}\in V_{x,1}$ such that $\ker D\rho^1\cap \ker D{\bar\rho^1}=\{0\}$ we have that $m_1(y)=0$ for all $y$ which are sufficiently close to $x$. Hence the set $\{x: m_i(x)=0\}$ is non-empty, open and $W_1^{0s}$-saturated. Hence $m_1(x)=m_1=0$ for all $x\in M$.

For any $x$ consider $(\rho^1,\rho^2), ({\bar \rho^1}, \bar\rho^2) \in V_{x}$ such that $\ker D\rho^1\cap \ker D{\bar\rho^1}=\{0\}$. Let $\euP^1=(\rho^1, {\bar \rho^1})$ and $\euP^2=(\rho^2, {\bar \rho^2})$. Because $D\rho^1(x)$ and $D\bar\rho^1(x)$ are linearly independent we have that $\euP^1$ is an invertible $C^1$-diffeomorphism on a small neighborhood of $x$ by the Inverse Function Theorem. By matching relations we have
$$
\euP^1=\euP^2\circ h
$$
and we obtain that $h^{-1}=(\euP^1)^{-1}\circ \euP^2$, which is $C^1$ on a small neighborhood of $x$. Since $x$ was an arbitrary point we have that $h^{-1}$ is $C^1$ along the unstable foliation.

To see that $h$ is also $C^1$ along the unstable foliation note that $h^{-1}=(\euP^1)^{-1}\circ \euP^2$ is a $C^1$ map which maps a 2-dimensional open set in the unstable leaf to a 2-dimensional open set. Then the measure of the image of $h^{-1}=(\euP^1)^{-1}\circ \euP^2$ is positive and given by the integral of Jacobian of $h^{-1}$. Hence this Jacobian is non-zero on some open subset, which implies, again by the Inverse Function Theorem, that $h^{-1}$ a $C^1$ diffeomorphism when restricted to this subset. Hence, on this set we have $m_2(x)=0$ and because the set $\{x: m_2(x)=0\}$ is $W^{0s}_2$-saturated we can conclude that $h$ is $C^1$ along every unstable leaf by exactly the same argument we used for $h^{-1}$ above. (Note that we couldn't argue completely symmetrically here because we did not assume existence of a periodic point with complex eigenvalues for $\phi_2^t$.) 


So, we have that $h$ is a $C^1$ diffeomorphism when restricted to any unstable leaf.  Finally to bootstrap the regularity of $h$ along the unstable leaves we can apply \cite[Proposition~3.3]{GRH3} to the time-$T$ maps $\phi_1^T$ and $\phi_2^T$ and conclude that $h$ is a smooth diffeomorphism when restricted to $W^u_1(p)$. By using dynamics we  have that $h|_{W^{0u}_1(p)}$ is also smooth. Then to obtain uniform smoothness along all unstable leaves we can use denseness of $W^{0u}_1(p)$  and repeat the arguments in the second half of Section~3 of~\cite{GRH3} almost verbatim adjusting to the flow case.

Recall that the flows $\phi_i^t$ are volume preserving. We already have that conjugacy $h$ is smooth along the unstable foliation. Hence the unstable Jacobians match at periodic points and invariance of volume gives matching of stable Jacobians as well. Then $h$ is smooth along the stable foliation by using the 1-dimensional de la Llave argument~\cite{dlL} and by applying the Journ\'e Lemma~\cite{J} we can conclude that $h$ is a smooth diffeomorphism.

\section{Proof of Theorem~\ref{thm_flows2} on rigidity of codimension one volume preserving Anosov flows}

Without loss of generality we can assume that the stable subbundle is one-dimensional, otherwise we can reverse the time direction.
Theorem~\ref{thm_flows2} is deduced from the Subbundle Theorem~\ref{thm_tech} and the following proposition.
\begin{proposition}
\label{prop}
Let $M$ be a closed manifold of dimension at least 4. Consider the space of  volume preserving Anosov flows on $M$ with $\dim E^s=1$. Then there exists a $C^1$ open and $C^\infty$ dense subset $\cV$ such that if $\phi^t\in\cV$ then the unstable bundle does not admit a non-trivial subbundle $E\subset E^u$, $\dim E>0$, such that $E$ is $D\phi^t$-invariant and integrates to a foliation which is jointly integrable with $W^s$.
\end{proposition}

We first explain how this proposition implies Theorem~\ref{thm_flows2} and then prove the proposition.

We have conjugate flows $\phi_1^t$ and $\phi_2^t$ with $\phi_1^t\in\cV$, where $\cV$ is given by the above proposition. Applying the Subbundle Theorem yields a subbundle $E_1\subset E_1^u$ which is invariant and integrates jointly with $E_1^s$. If $\dim E_1>0$ then we immediately arrive at a contradiction since Proposition~\ref{prop} assets that no such subbundle exists.

Hence we only need to consider the case when $\dim E_1=0$. In this case item~4 of Theorem~\ref{thm_tech} gives smoothness of the conjugacy $h$ along the unstable foliation. Now we can finish the proof in the same way as the proof of Theorem~\ref{thm_flows}: matching of unstable Jacobians at periodic points and invariance of volume yield matching of stable Jacobians at periodic points; then smoothness of $h$ follows from work of  de la Llave argument~\cite{dlL} and Journ\'e Lemma~\cite{J}.

\begin{remark}
Note that the volume preserving assumption is used two times. First time is to claim existence of a non-trivial matching pair, which comes from matching of simple PCFs. These PCFs are $C^1$ regular due to $C^1$ regularity of the stable holonomy, which, in turn needs the flow to be conservative. We use it the second time to conclude matching of stable Jacobians from matching of unstable Jacobians.
\end{remark}

\begin{proof}[Proof of Proposition~\ref{prop}] We consider the space of volume preserving Anosov diffeomorphisms on $M$ with $C^1$ topology. We begin by observing that it is sufficient to find a $C^1$ open and $C^\infty$-dense subset with posited property in a sufficiently small $C^1$ neighborhood of a given flow $\phi_0^t$. Then the set $\cV$ is given by taking the union of all such subsets.

Hence we fix a codimension one volume preserving Anosov flow $\phi_0^t$. Let $p$ and $q$ be points which belong to distinct periodic orbits and which are heteroclinically related. Pick a heteroclinic point $r\in W_{loc}^s(p)\cap W^{0u}(q)$. Now fix a small $C^1$ neighborhood $\cU$ of $\phi_0^t$ so that $p$, $q$ and $r$ admit continuations for $\phi^t\in\cU$. We define the set $\cV\subset\cU$ in the following way. A flow $\phi^t\in\cV$ if it satisfies the following conditions:
\begin{enumerate}
\item $D\phi^{T(p)}|_{E^u(p)} \colon E^u(p)\to E^u(p)$ admits only finitely many invariant linear subspaces $F_\alpha\subset E^u(p)$, $\alpha\in\cA$, of dimension $\ge 1$;
\item $D\,Hol^s_{r,p}(E^u(r))\nsupset F_\alpha$ for all $\alpha\in\cA$, where $Hol^s_{r,p}\colon W^{0u}_{loc}(r)\to W^{0u}_{loc}(p)$ is the stable holonomy which takes $r$ to $p$.
\end{enumerate}

First, recall that a linear map admits infinitely many distinct invariant subspaces if and only if it has a repeated eigenvalue with at least two linearly independent eigenvectors. This property is closed and, hence, condition~1 above is a $C^1$ open condition. The condition~2 is also $C^1$ open because the holonomy map $Hol^s_{r,p}$ varies continuously in $C^1$ topology. Therefore we have that $\cV\subset\cU$ is $C^1$ open.

Also if $\phi^t\in \cU$ admits a non-trivial subbundle $E\subset E^u$, $\dim E>0$, such that $E$ is $D\phi^t$-invariant and integrates to a foliation which is jointly integrable with $W^s$, then stable holonomy preserves the foliation of $E$ and, in particular, we have $D\,Hol^s_{r,p}(E(r))=E(p)$. Hence we have $D\,Hol^s_{r,p}(E^u(r))\supset E(p)$ and also, by invariance, $E(p)=F_\alpha$ for some $\alpha$. Hence, by condition~2 we can conclude that indeed $\phi^t\notin\cV$.

Hence, to finish the proof of the proposition it only remains to check that $\cV$ is $C^\infty$ dense in $\cU$. It is well known that eigenvalues at a periodic point can be independently perturbed via a $C^\infty$ small perturbation, which means that the condition~1 is $C^\infty$ dense. To check that the condition~2 is also $C^\infty$ dense we will need the following lemma.
\begin{lemma}
\label{lemma_main}
If $\phi^t\in\cU$ and satisfies condition 1 then it admits arbitrarily $C^\infty$-small reparametrization $\phi^t_\rho$ whose stable holonomy verifies condition~2, that is, $D\,Hol^s_{r,p}(E^u(r))\nsupset F_\alpha$ for all $\alpha\in\cA$.
\end{lemma}

Once this lemma is established $C^\infty$, density of condition~2 follows easily. Indeed, let $\Omega$ be the invariant volume for $\phi^t$, then the reparametrization $\phi^t_\rho$ preserves $\Omega_\rho$, which is close to $\Omega$ in $C^\infty$ topology ($\Omega_\rho$ is given by rescaling $\Omega$ by the inverse of the function which rescales the vector field and then normalizing to total volume 1). By employing Moser's trick we have a diffeomorphism $h$, $C^\infty$ close to $id_M$, such that $h^*\Omega_\rho=\Omega$. Let $\hat\phi_\rho^t=h^{-1}\circ \phi^t_\rho\circ h$. Then we have $\hat\phi_\rho^{t*}\Omega=\Omega$ and   $\hat\phi_\rho^t$ is close to $\phi^t$ in $C^\infty$ topology. Further, since $\hat\phi_\rho^t$ is smoothly conjugate to $\phi_\rho^t$, we still have that condition~2 holds for $\hat\phi_\rho^t$. Thus we have $C^\infty$ small perturbation $\hat\phi^t_\rho$ of $\phi^t$ which is $\cV$ and the proof of proposition is complete modulo the proof of Lemma~\ref{lemma_main}.
\end{proof}

\begin{proof}[Proof of Lemma~\ref{lemma_main}]
Consider a smooth transversal $\cT$ through the periodic point $p$ such that $W^s_{loc}(p), W^u_{loc}(p)\subset \cT$. Then locally $\phi^t$ is given as a suspension flow.
We denote by $f\colon\cT\cap f^{-1}(\cT) \to f(\cT)$ the local first return map and by $\rho_0\colon\cT\cap f^{-1}(\cT)\to\R_+$ the roof function. The reparametrizations $\phi^t_\rho$ are defined by changing the roof function from $\rho_0$ to $\rho_0+\rho$.  Fix a small ball $B\subset\cT$ which is centered at $f(r)$ and does not contain $r$ and $f^2(r)$. We impose the following conditions on $\rho$:
\begin{enumerate}
\item $\rho_0+\rho>0$;
\item $\rho(x)=0$, when $x\notin B$;
\item $\rho(x)=0$, when $x\in W^s_{loc}(p)$;
\end{enumerate}
Our goal is to find a $\rho$ which can be arbitrarily $C^\infty$ small, which implies that $\phi^t_\rho$ is arbitrarily close to $\phi^t$, and such that condition~2 holds for $\phi^t_\rho$. 

\begin{figure}[h]
  \includegraphics{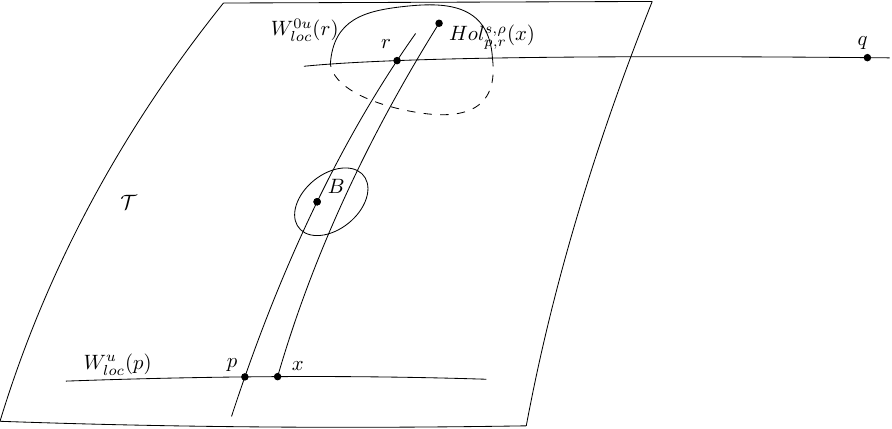}
  \caption{}
\end{figure}

We begin by making the following observations. Since $\phi^t_\rho$ is a reparametrization of $\phi^t$ the weak stable and weak unstable manifolds remain the same, while stable and unstable manifolds can change, that is adjust along the flow direction. However property~2 above implies that $W^s_{loc}(p)$ remains the same, and hence the heteroclinic point $r$ remains a heteroclinic point for $\phi^t_\rho$. Furthermore, $W^u_{loc}(q)$ and it's iterates $\phi^t(W^u_{loc}(q))$ are still unstable manifolds for as long as $\phi^t(W^u_{loc}(q))$  remains disjoint with $B$. In particular, since $r\notin B$, we have that $W^u_{loc}(r)$ and, hence, $E^u(r)$ remains the same for reparametrizations we consider. However the reparametrization affects the stable foliation near $W^s_{loc}(p)$ and hence affect the stable holonomy. We conclude that in order to verify condition 2:
$$
 D\,Hol^s_{r,p}(E^u(r))\nsupset F_\alpha,  \alpha\in\cA
 $$
 we only need to study the effect of $\rho$ on the derivative of the corresponding holonomy $D\,Hol^{s,\rho}_{r,p}\colon E^{0u}(r)\to E^{0u}(p)$ as both $E^u(r)$ and the finite collection of invariant subspaces $F_\alpha$ remain unchanged for all $\rho$ in the class of reparametrizations described above. Note that we began to use superscript $\rho$ to indicate dependence of holonomy map on the reparamentrization.
 
 In order to derive an explicit formula for  $D\,Hol^{s,\rho}_{p,r}$ we introduce a coordinate system on the neighborhood of $\cT$. These coordinates are induced by local foliations $W^{0s}_{loc}$, $W^{0u}_{loc}$ and by $\cT$. Specifically, we identify $p$ with $0$ and use $x$-coordinate on $W^{u}_{loc}(p)$, $y$-coordinate on $W^{s}_{loc}(p)$. Then we say that a point $z$ has coordinates $(x,t,y)$ if 
 $$
 z=\phi^t(W^{0s}_{loc}(x)\cap W^{0u}_{loc}(y)\cap \cT)=W^{0s}_{loc}(x)\cap W^{0u}_{loc}(y)\cap\phi^t(\cT)
 $$
 Because weak foliations are $C^1$ this coordinate system is also $C^1$. Heteroclinic point $r$ has coordinates $(0,0, y_r)$. Weak stable leaves are given by $x=const$ and hence the stable leaves for $\phi^t$ are given by graphs of functions of $y$-coordinate, that is,
 $$
 W^s_{loc}(x,0,0)=\{(x,y, T(x,y)): \, y\in W^s_{loc}(p)\}
 $$
 Similarly, for reparametrized flows $\phi^t_\rho$ we have 
$$
 W^s_{loc,\phi^t_\rho}(x,0,0)=\{(x,y, T^\rho(x,y)): \, y\in W^s_{loc}(p)\}
 $$
 Because $W^s_{loc}(p)$ does not change under reparametrization we have $T^\rho(o,y)=0$ and the holonomy $Hol_{p,r}^{s,\rho}\colon W^{0u}_{loc}(p)\to W^{0u}_{loc}(r)$ has the form
 $$
 Hol_{p,r}^{s,\rho}(x,t)=(x, t+T^\rho(x,y_r))
 $$
 Accordingly,
 $$
 D\,Hol^{s,\rho}_{p,r}=
 \begin{pmatrix}
 Id_x & 0\\
 D_x T^\rho_{(0,y_r)} & 1
 \end{pmatrix}
 $$
 where $D_x$ denotes the differentail with respect to the $x$-variable. For $\phi^t$, when $\rho=0$, the differential is given by the same expression with  $D_x T_{(0,y_r)}$ in the corner. We will derive the following formula for $D_xT^\rho_{(0,y_r)}$.
 \begin{claim}
 \label{claim_44}
 Given a reparametrization $\phi^t_\rho$, where $\rho$ satisfies conditions 1, 2 and 3 we have
 $$
  D_x T^\rho_{(0,y_r)}= D_x T_{(0,y_r)}+D_x\rho_{f(0,y_r)} D_xf_{(0,y_r)}
 $$
 \end{claim}
 Using this formula we can finish the proof of the lemma. We have
 \begin{multline*}
  D\,Hol^{s,\rho}_{p,r}(E^u(p))=\{(v, D_x T^\rho(v)):\, v\in E^u(p)\}\\
  =\{(v, D_x T(v))+(0, D_x\rho_{f(0,y_r)}(D_xf_{(0,y_r)}(v))):\, v\in E^u(p)\}
 \end{multline*}
 Note that $D_xf_{(0,y_r)}$ is a linear isomorphism and $D_x\rho_{f(0,y_r)}$ can be any linear functional with small norm because $f(0,y_r)=f(r)$ which is in the support of $\rho$. (Recall that we didn't impose any restriction on derivatives of $\rho$ in the unstable direction, except that we need $\rho$ to be $C^\infty$ close to 0.) Hence,  from above formula we have that $D\,Hol^{s,\rho}_{p,r}(E^u(p))$ can be any codimension one linear subspace in the neighborhood of $\{(v, D_x T(v)):\, v\in E^u(p)\}=D\,Hol^{s}_{p,r}(E^u(p))$.
 Therefore, as we vary $\rho$ the spaces $D\,Hol^{s,\rho}_{p,r}(E^u(p))\cap E^u(r)$ form an open set in the Grassmannian space of codimension one subspaces of $E^u(r)$. 
 
 Because $Hol^{s,\rho}_{r,p}$ is the inverse of $Hol^{s,\rho}_{p,r}$ the matrix for $D\,Hol^{s,\rho}_{r,p}$ has the same lower-triangular structure. The space $E^u(r)$ has the form $\{(v, A(v))\}$, hence, by the same token we also have the same conclusion for $D\,Hol^{s,\rho}_{r,p}(E^u(r))$. That is, as we vary $\rho$ the spaces $D\,Hol^{s,\rho}_{r,p}(E^u(r))\cap E^u(p)$ form an open set in the Grassmannian space of codimension one subspaces of $E^u(p)$. Pick a $\rho$ such that $D\,Hol^{s,\rho}_{r,p}(E^u(r))\cap E^u(p)$ is a generic codimension one subspace of $E^u(p)$. Then it cannot contain any of the (finitely many) invariant subspace $F_\alpha\subset E^u(p)$. Indeed since this codimension one subspace is generic we have the following formula for the dimension
 $\dim(D\,Hol^{s,\rho}_{r,p}(E^u(r))\cap E^u(p)\cap F_\alpha=\dim F_\alpha-1$, which implies that it cannot contain $F_\alpha$. (Recall that $\dim F_\alpha\ge 1$.)
\end{proof}

It remains to prove the formula for the differential of the holonomy.
\begin{proof}[Proof of Claim~\ref{claim_44}]
We need to compare $T^\rho(x,y_r)$ to $T(x,y_r)$ for small $x$. The points $(x,0,0)$ and $(x, T(x,y_r), y_r)$ belong to the same local stable manifolds and are asymptotic in the future under $\phi^t$, while the points $(x,0,0)$ and $(x, T^\rho(x,y_r), y_r)$ are asymptotic in the future under $\phi^t_\rho$. The forward orbits of these two pairs of points are exactly the same except when they cross the ball $B\subset\cT$, where the reparametrization is supported. First time this happens under the first iteration of the return map $f$ since $f(x,y_r)\in B$ for small $x$. Accordingly the ``stable manifold adjusts'' so that the points remain forward asymptotic:
$$
T^\rho(x,y_r)\approx T(x,y_r)+\rho(f(x,y_r))
$$
For the next return to $B$ to happen the point should leave the neighborhood of the periodic orbit of $p$ first, that is, reach the domain where the first return map $f\colon\cT\cap f^{-1}(\cT)\to\cT$ is no longer well-defined, which will take a very long time because $x$ is small. Once the points $\phi^t(x,0,0)$ and $\phi^t(x, 0, y_r)$ leave the neighborhood of the orbit of $p$ they would be very close and will have common future returns to $B$. More precisely, let $\hat B$ be a small neighborhood of $B$ in $\cT$ and denote by $(x_n,y_n)$ $n$-th return of $\phi^t(x,0, y_r)$ to $\hat B$. (We are considering $\hat B$ instead of $B$ so that if the orbit of $(x,0, y_r)$ misses $\hat B$ then the orbit of $(x,0,0)$ misses $B$ and vice versa. It could happen that $\phi^t(x,0, y_r)$ returns to $B$ near its boundary and $\phi^t(x,0, 0)$ misses B.) Then the corresponding return of $\phi^t(x,0, 0)$ has the form $(x_n, \bar y_n)\in\cT$. The reparametrization affects the stable manifolds only at these returns by adjusting with corresponding values of $\rho$. And for $(x,0,0)$ and $(x, T^\rho(x,y_r), y_r)$ to be asymptotic in forward time we must have the following relation
$$
T^\rho(x,y_r)= T(x,y_r)+\rho(f(x,y_r))+\sum_{n\ge 2} \rho(x_n,y_n)-\rho(x_n,\bar y_n)
$$
We will now show that 
$$
\sum_{n\ge 2} \rho(x_n,y_n)-\rho(x_n,\bar y_n)=o(\|x\|)
$$
Then differentiating with respect to $x$ at $(0,y_r)$ the above relation immediately yields the formula posited in the claim.

To estimate the series we first estimate the distance between $(x_2,y_2)$ and $(x_2,\bar y_2)$. As we mentioned before, for the second return to $\hat B$ to happen the orbit needs to leave the neighborhood of the orbit of $p$. So let $N$ be the largest interger such that $f^N(x, y_r)\in\cT$ is well-defined. After the time $\approx N\textup{per}(p)$ the point could make the second return to $B$. 

Denote by $\mu>1$ the strongest expansion of $f$ along in the unstable direction and by $\lambda<1$ the weakest expansion of $f$ in the stable direction. Because $f$ is volume preserving and the unstable subbundle has dimension $\ge 2$ we have that $\lambda\mu^{\kappa}=1$ with $\kappa>1$. (By choosing $\cT$ small we can assume that $f$ is almost linear and then this relation becomes obvious, or we can employ the standard adapted metric construction to ensure that $\lambda$ and $\mu$ satisfy this relation with $\kappa>1$.) Then we have
$$ \|x\|\mu^N>C_1$$
because the forward orbit leaves $\cT$ in $N$ steps. Hence we have 
$$
d((x_2,y_2),(x_2,\bar y_2))\le d(f^N(x,y_r), f^N(x,0))\le C_2\lambda^N=\frac{C_2}{\mu^{\kappa N}}<C_3\|x\|^\kappa
$$
It remains to notice that each subsequent return to $\hat B$ will take some large uniform time and, hence, $d((x_{n+1},y_{n+1}),(x_{n+1},\bar y_{n+1}))<\beta d((x_n,y_n),(x_n,\bar y_n))$ with $\beta\in(0,1)$. We conclude 
\begin{multline*}
\left|\sum_{n\ge 2} \rho(x_n,y_n)-\rho(x_n,\bar y_n)\right|\le  |\rho|_{C^1} \sum_{n\ge 2} d((x_n,y_n), (x_n,\bar y_n))\\
\le |\rho|_{C^1} d((x_2,y_2),(x_2,\bar y_2)) \sum_{n\ge 2}\beta^{n-2}\le C  |\rho|_{C^1} \|x\|^\kappa=o(\|x\|)
\end{multline*}
which completes the proof of the claim.
\end{proof}

\end{document}